\documentclass[12pt]{amsart}
\usepackage{amsmath,amssymb}
\usepackage[T1]{fontenc}
\usepackage{tikz}

\newcommand{\Ia }{\mathbb I}
\newcommand{\Xa}{\mathbb{X}}
\newcommand{\Ya}{\mathbb{Y}}
\newcommand{\Za}{\mathbb{Z}}

\newcommand{\Ce }{\mathcal C}
\newcommand{\Xe }{\mathcal X}

\renewcommand{\int}{\operatorname{Int}}

\newlength{\dlugosc}
\newlength{\kolko}
\newtheorem{thm}{Theorem}
\newtheorem{pro}[thm]{Proposition}
\newtheorem{lem}[thm]{Lemma}

\newtheorem{cor}[thm]{Corollary}

\title{On the center of distances}
\thanks{The first and the third author acknowledge support from GA\v{C}R project 16-34860L and RVO: 67985840.}

\author{Wojciech Bielas}
\address{Institute of Mathematics, University of Silesia, Bankowa 14, 40-007 Katowice, Poland; Institute of Mathematics of the Czech Academy of Sciences, \v{Z}itn\'a 25, 115 67 Praha 1, Czech Republic}
\email{wojciech.bielas@us.edu.pl}

\author{Szymon Plewik}
\address{Institute of Mathematics, University of Silesia, Bankowa 14, 40-007 Katowice}
\email{plewik@math.us.edu.pl}

\author{Marta Walczy\'nska}
\address{Institute of Mathematics, University of Silesia, Bankowa 14, 40-007 Katowice, Poland; Institute of Mathematics of the Czech Academy of Sciences, \v{Z}itn\'a 25, 115 67 Praha 1, Czech Republic}
\email{marta.walczynska@us.edu.pl}

\begin{document}

\begin{abstract}
In this paper we introduce the notion of the center of distances of a metric space, which is required for a generalization of the theorem by J. von Neumann about permutations of two sequences with the same set of cluster points in a compact metric space.
Also, the introduced notion is used to study sets of subsums of some sequences of positive reals, as well for some impossibility proofs. We compute the center of distances of the Cantorval, which is the set of subsums of the sequence $\frac34, \frac12, \frac3{16}, \frac18, \ldots , \frac3{4^n}, \frac2{4^n}, \ldots$, and also for some related subsets of the reals.
\end{abstract}

\maketitle

\section{Introduction}\label{s1}

The center of distances seems to be an elementary and natural notion which, as we think, has not been studied in the literature.
It is an intuitive and natural concept which allows us to prove a generalization of the theorem about permutations of two sequences with the same set of cluster points in a compact metric space, see Theorem \ref{p1}.
We have realized that the computation of the center of distances---even for well-known examples---is not an easy task because it requires skillful use of fractions.
We have only found a few algorithms which enable computing centers of distances, see Lemmas \ref{lsk} and \ref{l7}.
So, we present the use of this notion for impossibility proofs, i.e., to show that a given set cannot be the set of subsums, for example Corollary \ref{c14}.

We refer the readers to the paper \cite{ni}, as {it is} a good introduction to facts about the set of subsums of a given sequence, cf. also papers \cite{bg} and \cite{jo} as well as others cited therein.
In several papers, the set of all subsums of the sequence $\frac34, \frac12, \frac3{16}, \frac18, \ldots , \frac3{4^n}, \frac2{4^n}, \ldots$, i.e., the set $\Xa$ consisting of all sums
\[\textstyle
\sum_{n\in A}\frac2{4^n} + \sum_{n\in B} \frac3{4^n},
\]
where $A$ and $B$ are arbitrary subsets of positive natural numbers, is considered.
J. A. Guthrie and J. E. Nymann, see \cite{gn} and compare \cite{ns} and \cite[p. 865]{ni}, have shown that $[\frac34, 1]\subset \Xa$.
But, as it can be seen in Corollary \ref{c6}, we get $ [\frac23,1] \subset \Xa$.
For these reasons, we felt that the arithmetical properties of $ \Xa$ are not well known and described in the literature.
Results concerning these properties are discussed in: Propositions \ref{p2}, \ref{p3} and \ref{p4}; Corollary \ref{c9}; Theorems \ref{t12}, \ref{12}, \ref{13}, \ref{t11} and \ref{14}; and also they are presented in Figures \ref{fig:an_approximation_of_the_Cantorval}, \ref{fig:the_arrangement_of_affine_copies_of_D} and \ref{fig:the_correspondence_of_gaps}.

\section{A generalization of a theorem by J. von Neumann} \label{s2}

Given a metric space $X$ with the distance $d$, consider the set
\[\textstyle
S(X) = \{\alpha\colon \forall_{x \in X} \exists_{y \in X} d(x,y)=\alpha \},
\]
which will be called the \textit{center of distances} of $X$.
Suppose that sequences $\{ x_n\colon n\in \omega\} $ and $\{ y_n\colon n\in \omega\} $ have the
same set of cluster points $C$.
For this, J. von Neumann \cite{ne} proved that there exists a permutation $\pi \colon \omega \to \omega $ such that
$\lim_{n\to +\infty} d(x_n, y_{\pi(n)}) =0$.
Proofs of the above statement can be found in \cite{ha} and \cite{yo}.
However, we would like to present a slight generalization of this result.

\begin{thm} \label{p1}
Suppose that sequences $\{a_n\}_{n \in \omega} $ and $\{b_n\}_{n \in \omega} $ have the
same set of cluster points $ C \subseteq X$, where $(X,d)$ is a compact metric space. If $\alpha \in S(C)$, then there exists a permutation $\pi \colon \omega \to \omega $ such that
$\lim_{n\to +\infty} d(a_n, b_{\pi(n)}) =\alpha$.
\end{thm}

\begin{proof}
We use the back-and-forth method, which was developed in \cite[p. 35--36]{hu}.
Given $\alpha \in S(C)$, we shall renumber $\{b_n\}_{n \in \omega}$ by establishing a permutation $\pi\colon \omega \to \omega$ such that \[\textstyle \lim_{n\to +\infty} d(a_n, b_{\pi(n)}) =\alpha .\]
Put $\pi (0)= 0$ and assume that values $\pi (0), \pi (1), \ldots, \pi (m-1)$ and inverse values $\pi^{-1} (0), \pi^{-1} (1), \ldots, \pi^{-1} (m-1)$ are already defined. We proceed step by step as follows.
 
If $\pi (m)$ is not defined, then take points $ x_m, y_m \in C$ such that $d(a_m, x_m)= d(a_m, C) $ and $d(x_m, y_m)=\alpha$.
Then choose $b_{\pi(m)}$ to be the first element of $\{b_n\}_{n \in \omega}$ not already used such that $d(y_m, b_{\pi(m)})< \frac{1}{m}$.
 
But, if $\pi^{-1} (m)$ is not defined, then take points $ p_m, q_m \in C$ such that $d(b_m, q_m) = d(b_m, C) $ and $d(p_m, q_m)=\alpha$.
Choose $a_{\pi^{-1}(m)}$ to be the first element of $\{a_n\}_{ n \in \omega}$ not already used such that $d(p_m, a_{\pi^{-1}(m)})< \frac{1}{m}$.
 
The set $C\subseteq X$, as a closed subset of a compact metric space, is compact.
Hence required points $x_m$, $y_m$, $p_m$ and $q_m$ always exist and also
\[\textstyle
\lim_{n\to +\infty} d(a_n, C)= 0 = \lim_{n\to +\infty} d(b_n, C).
\]
It follows that $ \alpha = d(x_m, y_m) = d(p_m, q_m) = \lim_{n\to +\infty} d(a_n, b_{\pi(n)})$.
\end{proof}

As we have seen, the notion of the center of distances appears in a natural way in the context of metric spaces.
Though the computation of the center of distances is not an easy task, it can be done for important examples giving further information about those objects.

\section{On the center of distances and the set of subsums}

Given a metric space $X$, observe that $0\in S(X)$ and also, if $X\subseteq [0, +\infty)$ and $0\in X$, then $S(X) \subseteq X$.

If $\{a_n\colon n\in\omega\}$ is a sequence of reals, then the set
\[\textstyle
X=\{\sum_{n\in A} a_n\colon A \subseteq \omega \}
\]
is called the \textit{set of subsums} of $\{a_n\}$.
In this case, we have $d(x,y)=|x-y|$.
If $X$ is a subset of the reals, then any maximal interval $(\alpha, \beta)$ disjoint from $X$ is called an $X$-\textit{gap}.
Additionally, when $X$ is a closed set, then any maximal interval $[\alpha,\beta]$ included in $X$ is called an $X$-\textit{interval}.

\begin{pro}\label{I}
If $X$ is the set of subsums of a sequence $\{a_n\}_{n\in\omega}$, then $a_n\in S(X)$, for all $n\in\omega$.
\end{pro}

\begin{proof}
Suppose $x=\sum_{n \in A} a_n \in X$. If $n\in A$, then $x-a_n\in X$ and $d(x, x-a_n)= a_n$.
When $n\notin A$, then $x+a_n\in X$ and $d(x, x+a_n)= a_n$.
\end{proof}

In some cases, the center of distances of the set of subsums of a given sequence can be determined. For example, the unit interval is the set of subsums of $\{ \frac1{2^n} \}_{n>0} $. So, the center of distances of $\{ \frac1{2^n} \}_{n>0} $ is equal to $[0,\frac12 ]$.

\begin{lem}\label{lsk}
Assume that $(\lambda\cdot [0,b))\cap X=\lambda\cdot X$, for a number $\lambda>0$ and a set $X\subseteq [0,b)$.
If $x\in [0,b)\setminus X$ and $n \in \omega$, then $\lambda^n x\notin X$.
\end{lem}

\begin{proof}
Without loss of generality, assume that $X\cap(0,b)\not=\emptyset$.
Thus $\lambda \leqslant 1$, since otherwise we would get $ b < \lambda^m t \in X$, for some $t\in X$ and $m\in \omega$.
Obviously $x= \lambda^0x \in [0,b)\setminus X$. Assume that $\lambda^nx\notin X$, so we get
\[\textstyle
\lambda\cdot[0,b) \ni \lambda^{n+1}x=\lambda\cdot \lambda^nx\notin \lambda \cdot X=(\lambda\cdot [0,b))\cap X.
\]
Therefore $\lambda^{n+1}x\notin X$, which completes the induction step.
\end{proof}

\begin{thm}\label{2I}
If $q>2$ and $a\geqslant 0$, then the center of distances of the set of subsums of a geometric sequence $\{\frac{a}{q^n}\}_{n>0}$ consists of exactly the terms of the sequence $0,\frac aq,\frac a{q^2},\ldots$.
\end{thm}

\begin{proof}
Without loss of generality we can assume $a=1$.
The diameter of the set $X$ of subsums of the sequence $\{\frac1{q^n}\}_{n\geqslant 1}$ equals $\frac1{q-1} =\sum_{n\geqslant 1} \frac1{q^n}= \Xe (0)$ and we always have \[\textstyle c_n=\frac1{q^n}> \sum_{i\geqslant n+1} \frac1{q^i} = \frac1{q^n} \cdot \frac1{q-1}=\Xe(n) > c_{n+1}=\frac1{q^{n+1}}.\]
So, $ c_1=\frac1q \in X $ implies that no $t > c_1 $ belongs to $S(X)$.
Indeed, $c_1 -t <0$ and $c_1 +t > \frac2q > \frac1{q-1}.$
By Proposition \ref{I}, we get $\frac1q \in S(X)$.
The interval $(\frac1{q (q-1)},\frac1q )= (\Xe(1), c_1)$ is an $X$-gap.
So, $0\in X$ implies that no $t \in (\Xe(1), c_1) $ belongs to $S(X)$.

Now, suppose that $t \in (c_2,\Xe(1)]$ and $ t +\Xe(2) < c_1.$ Thus $\Xe(2)\in X$ implies that $t \notin S(X)$. Indeed, $\Xe(2) <c_2 <t$ implies $\Xe(2) - t <0$, and $\Xe(2) + c_2=\Xe(1)$ implies that the $X$-gap $(\Xe(1), c_1)$ has to contain $t +\Xe(2)$.

If $c_1 \leqslant t +\Xe(2)$ and $t \in (c_2,\Xe(1)]$, then the interval $( \Xe(1) - t, c_1 - t )$ is included in the interval $ [0, \Xe(2)]$. No $X$-gap of the length $\frac{q-2}{q(q-1)}= c_1 - \Xe(1)$ is 
contained in the interval $[0, \Xe(2)]$. Therefore, there exists $x\in X \cap ( \Xe(1) - t, c_1 - t )$---one can even find $x$ being the end of an $X$-gap, which implies that $t \notin S(X)$.
Again, by Proposition \ref{I}, we get $\frac1{q^2} \in S(X)$.
We proceed farther using Lemma \ref{lsk} with $\lambda = \frac1{q^n}.$
If we take advantage of the similarity of $X$ and $\frac1{q^n}\cdot X$ we get as a result that no $t\not= 0$ different from $\frac1{q^n}$, where $n>0$, belongs to $S(X)$.
\end{proof}

Note that, when we put $a = 2$ and $q = 3$, then Theorem \ref{2I} applies to the Cantor ternary set. But for $a\in \{2,3\}$ and $q=4$ this theorem applies to sets $\Ce_1$ and $\Ce_2$ which will be defined in Section \ref{X}.

\section{An example of a Cantorval}\label{X}

Following \cite[p. 324]{gn}, consider the set of subsums 
\[\textstyle \Xa = \left\{\sum_{n>0}\frac{x_n}{4^n}\colon \forall_{n} \; x_n \in \{0, 2,3,5\}\right\}.\]
Thus, $\Xa = \Ce_1 + \Ce_2$, where $ \Ce_1= \{ \sum_{n\in A}\frac{2}{4^{n}}\colon 0\notin A \subseteq \omega \} $
and $\Ce_2= \{ \sum_{n\in B}\frac{3}{4^{n}}\colon 0\notin B \subseteq \omega\}. $ 
Following \cite[p. 330]{mo}, because of its topological structure, one can call this set a \textit{Cantorval} (or an \textit{$\mathcal M$-Cantorval}).

Before discussing the affine properties of the Cantorval $\Xa$ we shall introduce the following useful notions. Every point $ x$ in $\Xa$ is determined by the sequence $ \{x_n\}_{n>0} $, where $x=\sum_{n>0}\frac{x_n}{4^n}$. Briefly, $x_n$ is called the 
$n$-th \textit{digit} of $x$. But the sequence $ \{x_n\}_{n>0}$ is called a \textit{digital representation} of the point $x\in \Xa$. Keeping in mind the formula for the sum of an infinite geometric series, we denote tails of series as follows:
\begin{itemize}
\item[]$\Ce_1(n)= \sum_{k=n+1} \frac2{4^k} = \frac23 \cdot \frac1{4^n}$ and $\Ce_1(0)=\frac23$; 
\item[] $\Ce_2(n)= \sum_{k=n+1} \frac3{4^k} = \frac1{4^n}$ and $\Ce_2(0)=1$;
\item[]$\Xe(n)= \sum_{k=n+1} \frac5{4^k} = \frac53 \cdot \frac1{4^n}$ and $\Xe(0)=\frac53$.
\end{itemize}

Since $\Xe(0) = \frac53$ we get $\Xa \subset [0, \frac{5}{3}]$.
The involution $h\colon \Xa \to \Xa$ defined by the formula \[\textstyle x\mapsto h(x)=\frac{5}{3} - x\] is the symmetry of $\Xa$ with respect to the point $\frac{5}{6} .$
In order to check this, it suffices to note that:
\[\textstyle\sum_{n\in A}\frac{2}{4^{n}} + \sum_{n\in B}\frac{3}{4^{n}}=x\in \Xa \Rightarrow \frac{5}{3} -x=\sum_{\substack{n\notin A\\ n>0}}\frac{2}{4^{n}} + \sum_{\substack{n\notin B\\ n>0}}\frac{3}{4^{n}}\in \Xa;
\]
and also that
\[\textstyle \frac{1}{2}+ \sum_{n>0}\frac{5}{4^{2n}}= \frac{3}{4}+ \sum_{n>0}\frac{5}{4^{2n+1}}=\frac{5}{6}\in \Xa.\]
So, we get $\Xa=\frac53 -\Xa$ and $\Xa=h[\Xa]$.
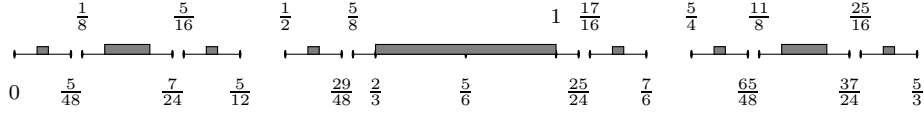
\begin{figure}[h]
\begin{tikzpicture}[xscale=1] \setlength{\dlugosc}{.15cm}\setlength{\kolko}{.04cm}
{\tiny
\draw (0\dlugosc,0) -- (5\dlugosc,0);
\draw (6\dlugosc,0) -- (14\dlugosc,0);
	\draw (15\dlugosc,0) -- (20\dlugosc,0);
	\draw (24\dlugosc,0) -- (29\dlugosc,0);
	\draw (30\dlugosc,0) -- (50\dlugosc,0); 
	\draw(51\dlugosc,0) -- (56\dlugosc,0);
	\draw (60\dlugosc,0) -- (65\dlugosc,0);
	\draw (66\dlugosc,0) -- (74\dlugosc,0);
	\draw (75\dlugosc,0) -- (80\dlugosc,0);
	
 \draw (0,-.5) node {$0$}; 
	\draw[x radius=.01,y radius=.05,fill] (0,0) circle;
	\draw (5\dlugosc,-.5) node {$\frac5{48}$}; 
	\draw[x radius=.01,y radius=.05,fill] (5\dlugosc,0) circle;
		\draw (6\dlugosc,+.5) node {$\frac1{8}$};
	\draw[x radius=.01,y radius=.05,fill] (6\dlugosc,0) circle;
		\draw (14\dlugosc,-.5) node {$\frac7{24}$}; 
	\draw[x radius=.01,y radius=.05,fill] (14\dlugosc,0) circle;
		\draw (15\dlugosc,+.5) node {$\frac5{16}$}; 
	\draw[x radius=.01,y radius=.05,fill] (15\dlugosc,0) circle;
		\draw (20\dlugosc,-.5) node {$\frac5{12}$};
	\draw[x radius=.01,y radius=.05,fill] (20\dlugosc,0) circle;
		\draw (24\dlugosc,+.5) node {$\frac1{2}$};
	\draw[x radius=.01,y radius=.05,fill] (24\dlugosc,0) circle;
		\draw (29\dlugosc,-.5) node {$\frac{29}{48}$};
	\draw[x radius=.01,y radius=.05,fill] (29\dlugosc,0) circle;
		\draw[x radius=.01,y radius=.05,fill] (30\dlugosc,0) circle;
	\draw (30\dlugosc,+.5) node {$\frac5{8}$}; 
	\draw (32\dlugosc,-.5) node {$\frac2{3}$}; 
	\draw (50\dlugosc,-.5) node {$\frac{25}{24}$};
	\draw[x radius=.01,y radius=.05,fill] (50\dlugosc,0) circle;
	\draw (51\dlugosc,+.5) node {$\frac{17}{16}$}; 
	\draw[x radius=.01,y radius=.05,fill] (51\dlugosc,0) circle;
		\draw (56\dlugosc,-.5) node {$\frac7{6}$};
	\draw[x radius=.01,y radius=.05,fill] (56\dlugosc,0) circle;
		\draw (60\dlugosc,+.5) node {$\frac5{4}$};
	\draw[x radius=.01,y radius=.05,fill] (60\dlugosc,0) circle;
		\draw (65\dlugosc,-.5) node {$\frac{65}{48}$};
	\draw[x radius=.01,y radius=.05,fill] (65\dlugosc,0) circle;
		\draw (66\dlugosc,+.5) node {$\frac{11}{8}$}; 
		\draw[x radius=.01,y radius=.05,fill] (66\dlugosc,0) circle;
			\draw (74\dlugosc,-.5) node {$\frac{37}{24}$};
		\draw[x radius=.01,y radius=.05,fill] (74\dlugosc,0) circle;
		\draw (75\dlugosc,+.5) node {$\frac{25}{16}$};
		\draw[x radius=.01,y radius=.05,fill] (75\dlugosc,0) circle;
		\draw (80\dlugosc,-.5) node {$\frac5{3}$}; 
		\draw[x radius=.01,y radius=.05,fill] (80\dlugosc,0) circle;
				\draw (40\dlugosc,-.5) node {$\frac56$}; 
		\draw[x radius=.01,y radius=0.05,fill] (40\dlugosc,0) circle;
		\draw (48\dlugosc,+.5) node {$1$}; 
		\draw[x radius=.01,y radius=.05,fill] (48\dlugosc,0) circle;
		\draw[x radius=.01,y radius=.05,fill] (32\dlugosc,0) circle;
		\draw[fill=gray] (2\dlugosc,0) rectangle (3\dlugosc,.1);
	
	\draw[fill=gray] (12\dlugosc,0) rectangle (8\dlugosc,.13);
	\draw[fill=gray] (17\dlugosc,0) rectangle (18\dlugosc,.1);
	\draw[fill=gray] (26\dlugosc,0) rectangle (27\dlugosc,.1);
	
	\draw[fill=gray] (48\dlugosc,0) rectangle (32\dlugosc,.13);
	\draw[fill=gray] (53\dlugosc,0) rectangle (54\dlugosc,.1);
	\draw[fill=gray] (62\dlugosc,0) rectangle (63\dlugosc,.1);
	\draw[fill=gray] (77\dlugosc,0) rectangle (78\dlugosc,.1);
	\draw[fill=gray] (68\dlugosc,0) rectangle (72\dlugosc,.13);
		}
		 \end{tikzpicture}
 \caption{An approximation of the Cantorval $\Xa \subset [0,\frac53]$.}\label{fig:an_approximation_of_the_Cantorval}
\end{figure}

In Figure \ref{fig:an_approximation_of_the_Cantorval} there are marked gaps $(\frac5{12},\frac12)$ and $(\frac{7}{6},\frac54)$, both of the length $\frac1{12}$.
Six gaps $(\frac5{48},\frac18)$, $(\frac7{24},\frac5{16})$, $(\frac{29}{48},\frac58)$, $(\frac{25}{24},\frac{17}{16})$, $(\frac{65}{48},\frac{11}{8})$ and $(\frac{37}{24},\frac{25}{16})$ 
have the length $\frac1{48}$.
The rest of gaps are shorter and have lengths not greater than $\frac1{192}$.
To describe intervals which lie in $\Xa$, we need the following.
Let
\[\textstyle
K_n =[\frac23, 1]\cap\{\sum_{i=1}^n \frac{x_i}{4^i}\colon \forall_i\; x_i \in \{0,2,3,5\} \}.
\]
We get $K_1 = \{\frac34\}$ and $ K_2= \{\frac{11}{16}, \frac{3}{4}, \frac{13}{16},\frac{7}{8}, \frac{15}{16}\}$.
Keeping in mind $\Ce_1(0)=\frac23$ and $\Ce_1(n) < \frac1{4^n}$, we check that $f^n_1=\frac3{4^n} + \sum_{i=1}^{n-1} \frac2{4^i}> \frac23$ is the smallest real number in $K_n$. Similarly, using $\Xe(n) = \frac53 \cdot \frac1{4^n} <\frac2{4^n}$ and $\Ce_2(0)=1$, check that $f^n_{|K_n|}=\sum_{i=1}^{n} \frac3{4^i}<1$ is the greatest real number in $K_n$. In fact, we have the following.

\begin{pro}\label{p2}
Reals from $K_n$ are distributed consecutively at the distance $\frac1{4^n}$, from $\frac3{4^n} + \sum_{i=1}^{n-1} \frac2{4^i}$ up to $\sum_{i=1}^{n} \frac3{4^i}$, in the interval $[\frac23, 1]$. Therefore $ |K_n| = \frac13(4^n-1) $ and $|K_{n+1}|= 4|K_n| +1$.
\end{pro}

\begin{proof}
Since $|K_1|=1$ and $|K_2|=5$, the assertions are correct in these cases.
Suppose that $K_{n-1}= \{ f^{n-1}_1, f^{n-1}_2, \ldots , f^{n-1}_{|K_{n-1}|}\},$ where
\[
f^{n-1}_1=\frac3{4^{n-1}}+\sum_{i=1}^{n-2} \frac2{4^i}, \mbox{ and } f^{n-1}_{j+1} - f^{n-1}_j = \frac1{4^{n-1}}
\]
for $0<j < |K_{n-1}|-1$; in consequence $f^{n-1}_{|K_{n-1}|}= \sum_{i=1}^{n-1} \frac3{4^i}$.
Consider the sum
\[\textstyle
K_{n-1} \cup (\frac2{4^n} + K_{n-1}) \cup (\frac3{4^n} + K_{n-1}) \cup (\frac5{4^n} + K_{n-1}),
\]
remove the point $\frac5{4^n} + \sum_{i=1}^{n-1} \frac3{4^i} >1$, and then add points $f^n_1=\frac3{4^n}+\sum_{i=1}^{n-1}\frac2{4^i}$ and $f^n_3=\frac5{4^n} + \sum_{i=1}^{n-1} \frac2{4^i}$.
We obtain the set \[K_n=\{ f^{n}_1, f^{n}_2, \ldots , f^{n}_{\frac13(4^n-1)}\} , \] which is what we need.
\end{proof}

\begin{cor}\label{c6}
The interval $[\frac23,1]$ is included in the Cantorval $ \Xa$.
\end{cor}

\begin{proof}
The union $\bigcup \{K_n\colon n>0\} $ is dense in the interval $[\frac23,1]$.
Hence, Proposition \ref{p2} implies the assertion.
\end{proof}

Note that it has observed that $[\frac34,1] \subset \Xa$, see \cite{gn} or compare \cite{ni}.
Since $\Xa$ is centrally symmetric with $\frac56$ as a point of inversion, this yields another proof of the above corollary.
However, our proof seems to be new and it is different than the one included in \cite{gn}.

Put $C_n= \frac1{4^n} \cdot \Xa=\Xa \cap [0, \frac5{3\cdot 4^n}]$, for $n\in \omega$. So, each $C_n$ is an affine copy of $\Xa$.

\begin{pro}\label{p3}
The subset $\Xa \setminus [\frac23, 1]\subset \Xa$ is the union of pairwise disjoint affine copies of $\Xa$. In particular, this union includes two isometric copies of $C_n=\frac1{4^n} \cdot \Xa$, for every $n>0$.
\end{pro}

\begin{proof}
The desired affine copies of $\Xa$ are $C_1$ and $\frac54 + C_1= h[C_1]$, $\frac12 + C_2$ and $h[\frac12 + C_2]$, and so on, i.e., $\sum_{i=1}^{n}\frac2{4^n} + C_{n+1}$ and $h[\sum_{i=1}^{n}\frac2{4^n} + C_{n+1}]$.
\end{proof}

\begin{pro}\label{p4}
The subset $\Xa \setminus ((\frac23, 1)\cup (\frac1{6}, \frac14) \cup (\frac{17}{12},\frac32))\subset \Xa$ is the union of six pairwise disjoint affine copies of $D= [0,\frac16] \cap \Xa$.
\end{pro}

\begin{proof}
The desired affine copies of $D= \frac14 \cdot (\Xa\cap [0,\frac23])$ lie as shown in Figure \ref{fig:the_arrangement_of_affine_copies_of_D}.
\begin{figure}[h]
\begin{tikzpicture}{\tiny \setlength{\dlugosc}{.15cm}\setlength{\kolko}{.04cm}
\draw[thick] (0\dlugosc,0) -- (5\dlugosc,0);
\draw[thick] (6\dlugosc,0) -- (14\dlugosc,0);
	\draw[thick] (15\dlugosc,0) -- (20\dlugosc,0);
	\draw[thick] (24\dlugosc,0) -- (29\dlugosc,0);
	\draw[thick] (30\dlugosc,0) -- (50\dlugosc,0); 
	\draw[thick](51\dlugosc,0) -- (56\dlugosc,0);
	\draw[thick] (60\dlugosc,0) -- (65\dlugosc,0);
	\draw[thick] (66\dlugosc,0) -- (74\dlugosc,0);
	\draw[thick] (75\dlugosc,0) -- (80\dlugosc,0);
 \draw (0\dlugosc,+.5) node {$0$};
 \draw[x radius=0.01,y radius=.05,fill] (0\dlugosc,0) circle;
		\draw[x radius=0.01,y radius=.05, fill] (5\dlugosc,0) circle;
		\draw[x radius=0.01,y radius=.05,fill] (6\dlugosc,0) circle;
		\draw[x radius=0.01,y radius=.05,fill] (14\dlugosc,0) circle;
		\draw[x radius=0.01,y radius=.05,fill] (15\dlugosc,0) circle;
	\draw (4\dlugosc,-.5) node {$D$};
	\draw (28\dlugosc,-.5) node {$\frac12+D$};
	\draw (64\dlugosc,-.5) node {$\frac54+D$};
	\draw (52\dlugosc,-.5) node {$h[\frac12+D]$};
	\draw (16\dlugosc,-.5) node {$h[\frac54+D]$};
	\draw (76\dlugosc,-.5) node {$h[D]$};	
	\draw (8\dlugosc,+.5) node {$\frac{1}{6}$}; \draw[x radius=0.01,y radius=.05,fill] (8\dlugosc,0) circle;
	\draw (12\dlugosc,+.5) node {$\frac{1}{4}$}; \draw[x radius=0.01,y radius=.05,fill] (12\dlugosc,0) circle;
	\draw (20\dlugosc,+.5) node {$\frac5{12}$}; \draw[x radius=0.01,y radius=.05,fill] (20\dlugosc,0) circle;
	\draw (24\dlugosc,+.5) node {$\frac1{2}$}; \draw[x radius=0.01,y radius=.05,fill] (24\dlugosc,0) circle;
		\draw[x radius=0.01,y radius=.05,fill] (29\dlugosc,0) circle;
		\draw[x radius=0.01,y radius=.05,fill] (30\dlugosc,0) circle;
	\draw (32\dlugosc,+.5) node {$\frac2{3}$}; \draw[x radius=0.01,y radius=.05,fill] (32\dlugosc,0) circle;
	\draw (48\dlugosc,+.5) node {$1$}; \draw[x radius=0.01,y radius=.05,fill] (48\dlugosc,0) circle;
		\draw[x radius=0.01,y radius=.05,fill] (50\dlugosc,0) circle;
		\draw[x radius=0.01,y radius=.05,fill] (51\dlugosc,0) circle;
	\draw (56\dlugosc,+.5) node {$\frac7{6}$}; \draw[x radius=0.01,y radius=.05,fill] (56\dlugosc,0) circle;
	\draw (60\dlugosc,+.5) node {$\frac5{4}$}; \draw[x radius=0.01,y radius=.05,fill] (60\dlugosc,0) circle;
	\draw[x radius=0.01,y radius=.05,fill] (65\dlugosc,0) circle;
	\draw[x radius=0.01,y radius=.05,fill] (66\dlugosc,0) circle;
	\draw[x radius=0.01,y radius=.05,fill] (74\dlugosc,0) circle;
	\draw[x radius=0.01,y radius=.05,fill] (75\dlugosc,0) circle;
	\draw (80\dlugosc,+.5) node {$\frac5{3}$}; \draw[x radius=0.01,y radius=.05,fill] (80\dlugosc,0) circle;
	\draw[fill=gray] (2\dlugosc,0) rectangle (3\dlugosc,.1);
	
	\draw[fill=gray] (12\dlugosc,0) rectangle (8\dlugosc,.13);
	\draw[fill=gray] (17\dlugosc,0) rectangle (18\dlugosc,.1);
	\draw[fill=gray] (26\dlugosc,0) rectangle (27\dlugosc,.1);
	
	\draw[fill=gray] (48\dlugosc,0) rectangle (32\dlugosc,.13);
	\draw[fill=gray] (53\dlugosc,0) rectangle (54\dlugosc,.1);
	\draw[fill=gray] (62\dlugosc,0) rectangle (63\dlugosc,.1);
	\draw[fill=gray] (77\dlugosc,0) rectangle (78\dlugosc,.1);
	\draw (68\dlugosc,+.5) node {$\frac{17}{12}$}; \draw[x radius=0.01,y radius=.05,fill] (68\dlugosc,0) circle;
	\draw (72\dlugosc,+.5) node {$\frac{3}{2}$}; \draw[x radius=0.01,y radius=.05,fill] (72\dlugosc,0) circle;	
	\draw[fill=gray] (68\dlugosc,0) rectangle (72\dlugosc,.13);
	\draw[x radius=0.01,y radius=.05,fill] (2\dlugosc,0) circle;
		\draw[x radius=0.01,y radius=.05,fill] (3\dlugosc,0) circle;

		\draw[x radius=0.01,y radius=.05,fill] (17\dlugosc,0) circle;
		\draw[x radius=0.01,y radius=.05,fill] (18\dlugosc,0) circle;	
	
		\draw[x radius=0.01,y radius=.05,fill] (26\dlugosc,0) circle;
	\draw[x radius=0.01,y radius=.05,fill] (27\dlugosc,0) circle;	
	
	\draw[x radius=0.01,y radius=.05,fill] (62\dlugosc,0) circle;
		\draw[x radius=0.01,y radius=.05,fill] (63\dlugosc,0) circle;	
	
	\draw[x radius=0.01,y radius=.05,fill] (77\dlugosc,0) circle;
	\draw[x radius=0.01,y radius=.05,fill] (78\dlugosc,0) circle;	
	
		\draw[x radius=0.01,y radius=.05,fill] (53\dlugosc,0) circle;
		\draw[x radius=0.01,y radius=.05,fill] (54\dlugosc,0) circle;	
	
	\draw[dashed] (0,-\dlugosc) -- (0,-1.5\dlugosc) -- (8\dlugosc,-1.5\dlugosc) -- (8\dlugosc,-\dlugosc);
	\draw[dashed] (12\dlugosc,-\dlugosc) -- (12\dlugosc,-1.5\dlugosc) -- (20\dlugosc,-1.5\dlugosc) -- (20\dlugosc,-\dlugosc);
	\draw[dashed] (24\dlugosc,-\dlugosc) -- (24\dlugosc,-1.5\dlugosc) -- (32\dlugosc,-1.5\dlugosc) -- (32\dlugosc,-\dlugosc);
	\draw[dashed] (48\dlugosc,-\dlugosc) -- (48\dlugosc,-1.5\dlugosc) -- (56\dlugosc,-1.5\dlugosc) -- (56\dlugosc,-\dlugosc);
	\draw[dashed] (60\dlugosc,-\dlugosc) -- (60\dlugosc,-1.5\dlugosc) -- (68\dlugosc,-1.5\dlugosc) -- (68\dlugosc,-\dlugosc);
	\draw[dashed] (72\dlugosc,-\dlugosc) -- (72\dlugosc,-1.5\dlugosc) -- (80\dlugosc,-1.5\dlugosc) -- (80\dlugosc,-\dlugosc);}
\end{tikzpicture}
\caption{The arrangement of affine copies of $D$.}\label{fig:the_arrangement_of_affine_copies_of_D}
\end{figure}
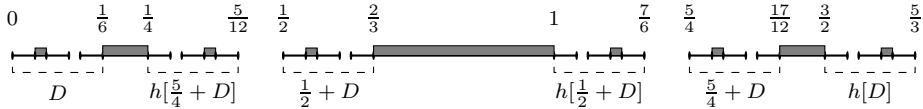
\end{proof}

\begin{cor}\label{c9}
The Cantorval $\Xa\subset [0,\frac53]$ has Lebesgue measure $1$.
\end{cor}

\begin{proof}
There exists a one-to-one correspondence between $\Xa$-gaps and $\Xa$-interval as it is shown in Figure \ref{fig:the_correspondence_of_gaps}.
\begin{figure}
\begin{tikzpicture}\setlength{\dlugosc}{.15cm}
{\tiny
 \draw[thick] (0\dlugosc,0) -- (5\dlugosc,0);
	\draw[thick] (6\dlugosc,0) -- (14\dlugosc,0);
	\draw[thick] (15\dlugosc,0) -- (20\dlugosc,0);
	\draw[thick] (24\dlugosc,0) -- (29\dlugosc,0);
	\draw[thick] (30\dlugosc,0) -- (50\dlugosc,0); 
	\draw[thick](51\dlugosc,0) -- (56\dlugosc,0);
	\draw[thick] (60\dlugosc,0) -- (65\dlugosc,0);
	\draw[thick] (66\dlugosc,0) -- (74\dlugosc,0);
	\draw[thick] (75\dlugosc,0) -- (80\dlugosc,0);
 \draw (0\dlugosc,-.5) node {$0$}; \draw[x radius=.01,y radius=.05,fill] (0\dlugosc,0) circle;
		\draw[x radius=.01,y radius=.05, fill] (5\dlugosc,0) circle;
		\draw[x radius=.01,y radius=.05,fill] (6\dlugosc,0) circle;
		\draw[x radius=.01,y radius=.05,fill] (14\dlugosc,0) circle;
		\draw[x radius=.01,y radius=.05,fill] (15\dlugosc,0) circle;
	\draw (8\dlugosc,+.5) node {$\frac{1}{6}$}; \draw[x radius=.01,y radius=.05,fill] (8\dlugosc,0) circle;
	\draw (12\dlugosc,-.5) node {$\frac{1}{4}$}; \draw[x radius=.01,y radius=.05,fill] (12\dlugosc,0) circle;
	\draw (20\dlugosc,-.5) node {$\frac5{12}$}; \draw[x radius=.01,y radius=.05,fill] (20\dlugosc,0) circle;
	\draw (24\dlugosc,-.5) node {$\frac1{2}$}; \draw[x radius=.01,y radius=.05,fill] (24\dlugosc,0) circle;
		\draw[x radius=.01,y radius=.05,fill] (29\dlugosc,0) circle;
		\draw[x radius=.01,y radius=.05,fill] (30\dlugosc,0) circle;
	\draw (32\dlugosc,+.5) node {$\frac2{3}$}; \draw[x radius=.01,y radius=.05,fill] (32\dlugosc,0) circle;
	\draw (48\dlugosc,+.5) node {$1$}; \draw[x radius=.01,y radius=.05,fill] (48\dlugosc,0) circle;
		\draw[x radius=.01,y radius=.05,fill] (50\dlugosc,0) circle;
		\draw[x radius=.01,y radius=.05,fill] (51\dlugosc,0) circle;
	\draw (56\dlugosc,-.5) node {$\frac7{6}$}; \draw[x radius=.01,y radius=.05,fill] (56\dlugosc,0) circle;
	\draw (60\dlugosc,-.5) node {$\frac5{4}$}; \draw[x radius=.01,y radius=.05,fill] (60\dlugosc,0) circle;
		\draw[x radius=.01,y radius=.05,fill] (65\dlugosc,0) circle;
		\draw[x radius=.01,y radius=.05,fill] (66\dlugosc,0) circle;
	\draw[x radius=.01,y radius=.05,fill] (74\dlugosc,0) circle;
	\draw[x radius=.01,y radius=.05,fill] (75\dlugosc,0) circle;
	\draw (80\dlugosc,-.5) node {$\frac5{3}$}; \draw[x radius=.01,y radius=.05,fill] (80\dlugosc,0) circle;
	\draw[fill=gray] (48\dlugosc,0) rectangle (32\dlugosc,.13);
	\draw[fill=gray] (12\dlugosc,0) rectangle (8\dlugosc,.13);
	\draw (68\dlugosc,-.5) node {$\frac{17}{12}$}; \draw[x radius=.01,y radius=.05,fill] (68\dlugosc,0) circle;
	\draw (72\dlugosc,+.5) node {$\frac{3}{2}$}; \draw[x radius=.01,y radius=.05,fill] (72\dlugosc,0) circle;	
	\draw[fill=gray] (68\dlugosc,0) rectangle (72\dlugosc,.13);
	\draw[x radius=.01,y radius=.05,fill] (2\dlugosc,0) circle;
		\draw[x radius=.01,y radius=.05,fill] (3\dlugosc,0) circle;	
	\draw[fill=gray] (2\dlugosc,0) rectangle (3\dlugosc,.1);
		\draw[x radius=.01,y radius=.05,fill] (17\dlugosc,0) circle;
		\draw[x radius=.01,y radius=.05,fill] (18\dlugosc,0) circle;	
	\draw[fill=gray] (17\dlugosc,0) rectangle (18\dlugosc,.1);
		\draw[x radius=.01,y radius=.05,fill] (26\dlugosc,0) circle;
		\draw[x radius=.01,y radius=.05,fill] (27\dlugosc,0) circle;	
	\draw[fill=gray] (26\dlugosc,0) rectangle (27\dlugosc,.1);
		\draw[x radius=.01,y radius=.05,fill] (62\dlugosc,0) circle;
		\draw[x radius=.01,y radius=.05,fill] (63\dlugosc,0) circle;	
	\draw[fill=gray] (62\dlugosc,0) rectangle (63\dlugosc,.1);
		\draw[x radius=.01,y radius=.05,fill] (77\dlugosc,0) circle;
		\draw[x radius=.01,y radius=.05,fill] (78\dlugosc,0) circle;	
	\draw[fill=gray] (77\dlugosc,0) rectangle (78\dlugosc,.1);
		\draw[x radius=.01,y radius=.05,fill] (53\dlugosc,0) circle;
	\draw[x radius=.01,y radius=.05,fill] (54\dlugosc,0) circle;	
	\draw[fill=gray] (53\dlugosc,0) rectangle (54\dlugosc,.1);
		\draw[<->,densely dotted] (58\dlugosc,.2) -- (58\dlugosc,.6) -- (70\dlugosc,.6) -- (70\dlugosc,.2);
	\draw[<->,densely dotted] (10\dlugosc,.2) -- (10\dlugosc,.6) -- (22\dlugosc,.6) -- (22\dlugosc,.2);
	\draw[<->,densely dotted] (2.5\dlugosc,.2) -- (2.5\dlugosc,.4) -- (5.5\dlugosc,.4) -- (5.5\dlugosc,.2);
	\draw[<->,densely dotted] (14.5\dlugosc,.2) -- (14.5\dlugosc,.4) -- (17.5\dlugosc,.4) -- (17.5\dlugosc,.2);
	\draw[<->,densely dotted] (26.5\dlugosc,.2) -- (26.5\dlugosc,.4) -- (29.5\dlugosc,.4) -- (29.5\dlugosc,.2);
	\draw[<->,densely dotted] (50.5\dlugosc,.2) -- (50.5\dlugosc,.4) -- (53.5\dlugosc,.4) -- (53.5\dlugosc,.2);
	\draw[<->,densely dotted] (62.5\dlugosc,.2) -- (62.5\dlugosc,.4) -- (65.5\dlugosc,.4) -- (65.5\dlugosc,.2);
	\draw[<->,densely dotted] (74.5\dlugosc,.2) -- (74.5\dlugosc,.4) -- (77.5\dlugosc,.4) -- (77.5\dlugosc,.2);
	}
\end{tikzpicture}
\caption{The correspondence between $\Xa$-gaps and $\Xa$-intervals.}\label{fig:the_correspondence_of_gaps}
\end{figure}
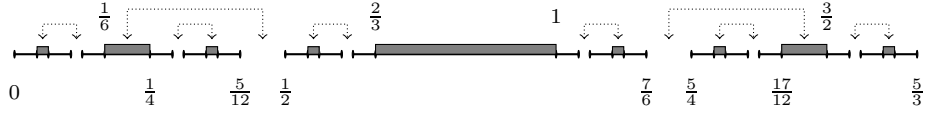
In view of Propositions \ref{p3} and \ref{p4}, we calculate the sum of lengths of all gaps which lie in $[0,\frac53] \setminus \Xa$ as follows: 
\[\textstyle
\frac16 + 6 \cdot \frac1{3\cdot 4^2 } + \frac18 \cdot \frac34 + \ldots + \frac18 \cdot \left(\frac34\right)^n + \ldots = \frac16 + \frac18 \sum_{n\geqslant 0} \left(\frac34\right)^n = \frac23.
\]
Since $\frac53 - \frac23 =1$ we are done.
\end{proof}

If we remove the longest interval from $\frac1{4^n}\cdot D$, then we get the union of three copies of $D$, each congruent to $\frac1{4^{n+1}}\cdot D$.
This observation---we used it above by default---is sufficient to calculate the sum of lengths of all $\Xa$-intervals as follows:
\[\textstyle
\frac13+\frac16+6\cdot\frac1{3\cdot4^2}+\ldots+\frac18\cdot\left(\frac34\right)^n+\ldots=1.
\]
Therefore the boundary $\Xa \setminus \int \Xa $ is a null set.

\section{Computing the center of distances}
In case of subsets of the real line we formulate the following lemma.

\begin{lem}\label{l7}
Given a set $C\subseteq [0,+\infty)$ disjoint from an interval $(\alpha, \beta)$, assume that $x\in [0, \frac{\alpha}{2}] \cap C$. Then the center of distances $S(C)$ is disjoint from the interval $(\alpha - x, \beta -x)$, i.e., 
$S(C)\cap ((\alpha, \beta) - x)=\emptyset$.
\end{lem}

\begin{proof}
Given $x\in [0, \frac{\alpha}{2}] \cap C$, consider $t \in (\alpha - x, \beta -x)$. We get \[x \leqslant \frac\alpha2 \leqslant \alpha - x < t < \beta -x. \]
Since $\alpha < x+t < \beta$, we get $x+t \notin C$, also $x<t$ implies $x-t \notin C$. Therefore, $x\in C$ implies $t\notin S(C)$.
\end{proof}

We will apply the above lemma by putting suitable $C$-gaps in place of the interval $(\alpha,\beta)$.
In order to obtain $t\notin S(C)$, we must find $x <t$ such that $x+t \in (\alpha, \beta)$ and $x\in C$.
For example, this is possible when $\frac\alpha2 < t < \alpha$ and the interval $[0,\alpha]$ includes no $C$-gap of the length greater than or equal to $\beta -\alpha$. But if such a gap exists, then we choose the required $x$ more carefully.
 
\begin{thm}\label{t12}
The center of distances of the Cantorval $\Xa$ consists of exactly the terms of the sequence $0,\frac34,\frac12,\ldots,\frac3{4^n},\frac2{4^n},\ldots$.
\end{thm}

\begin{proof}
We get $ \{\frac2{4^n}\colon n>0\} \cup \{\frac3{4^n}\colon n>0\}\subseteq S(\Xa)$, by Proposition \ref{I}. Also, the diameter of $\Xa$ is $\frac53$ and $\frac56 \in \Xa$, hence no $t>\frac56$ belongs to $S(\Xa)$.
We use Lemma \ref{l7} with respect to the gap $(\alpha, \beta)=(\frac76, \frac54). $ Keeping in mind the affine description of $\Xa$, we see that the set $\Xa\cap[0, \frac7{12}]$ has a gap $(\frac5{12},\frac12)$ of the length $\frac1{12}$. For $t\in(\frac7{12},\frac76)\setminus\{\frac34\}$, we choose {$x$ in $\Xa$} such that $x \in (\frac76 - t, \frac54 - t)$. So, if {$t\in(\frac7{12},\frac76)\setminus\{\frac34\}$}, then $t\notin S(\Xa)$.
Similarly using Lemma \ref{l7} with the gap $(\alpha, \beta)=(\frac{29}{48}, \frac{5}{8})$, we check that for {$t\in(\frac{29}{96},\frac7{12}]\setminus\{\frac12\}$} there exists {$x$ in $\Xa$} such that $x \in (\frac{29}{48} - t, \frac{5}{8} - t)$. Hence, if {$t\in(\frac{29}{96},\frac7{12}]\setminus\{\frac12\}$}, then $t\notin S(\Xa)$.
Analogously, using Lemma \ref{l7} with the gap $(\alpha, \beta)=(\frac{5}{12}, \frac{1}{2})$, we check that if $\frac5{24} < t \leqslant \frac{29}{96} < \frac5{12}$, then $t\notin S(\Xa)$.

So far, we have shown that numbers $\frac12$ and $\frac34$ are the only elements of $S(\Xa)\cap(\frac5{24}, +\infty)$.
For the remaining part of the interval $[0,+\infty)$ the proof proceeds by induction on $n$, since $\Xa$ and $\frac1{4^n}\cdot \Xa$ are similar.
\end{proof}

Denote $\Za = [0,\frac53] \setminus \int \Xa $. Thus the closure of a $\Xa$-gap is the $\Za$-interval and the interior of a $\Xa$-interval is the $\Za$-gap.

\begin{thm}\label{12}
The center of distances of the set $\Za$ is trivial, i.e., $S(\Za)= \{0\}$.
\end{thm}

\begin{proof}
If $\alpha >1$, then $\{1+\alpha, 1-\alpha\} \cap \Za = \emptyset$, hence $1\in \Za$ implies $\alpha \notin S(\Za)$.
If $\alpha \in \{\frac14, 1\}$, then $\frac{11}{24} \in (\frac5{12}, \frac12) \subset \Za$ implies $\alpha \notin S(\Za)$.
Indeed, the number $1+\frac{11}{24}$ belongs to the $\Za$-gap $(\frac{17}{12}, \frac32)$ and the number $\frac14+\frac{11}{24}$ belongs to the $\Za$-gap $(\frac{2}{3}, 1)$ and the number $\frac{11}{24} - \frac14$ belongs to the $\Za$-gap $(\frac16, \frac14)$.
Also $0\in \Za $ implies $(\frac23, 1)\cap S(\Za)=\emptyset,$	since $(\frac23, 1)$ is a $\Za$-gap.
For the same reason $\frac14 \in \Za$ implies $\frac23 \notin S(\Za)$ and $\frac13 \in \Za$ implies that no $\alpha \in (\frac13,\frac23)$ belongs to $S(\Za)$. Since $\frac16 <\frac{17}{32} -\frac13 <\frac14$ and $\frac23 <\frac{17}{32} +\frac13 <1$, then $\frac{17}{32} \in \Za$ implies $\frac13 \notin S(\Za)$. But, if $\alpha \in (\frac14, \frac13)$, then $\frac12 \in \Za$ implies $\alpha \notin S(\Za)$. Indeed, $\frac16 <\frac12 - \alpha <\frac14 $ and $\frac34 <\frac12 + \alpha <\frac56.$ So far, we have shown $S(\Za) \cap [\frac14, +\infty) =\emptyset$.
In fact, sets $[\frac1{4^{n+1}}, \frac1{4^n}]$ and $S(\Za)$ are always disjoint, since
\[\textstyle
\frac1{4^n} \cdot (\Za \cap [0,1]) = [0, \frac1{4^n}] \cap \Za.
\]
Therefore $\alpha \in [\frac1{4^{n+1}}, \frac1{4^n}] \cap S(\Za)$ implies $4^n \cdot \alpha \in [\frac1{4}, 1] \cap S(\Za)$, a contradiction. Finally, we get $S(\Za) = \{0\}$.
\end{proof}

Now, denote $\Ya=\Za \cap \Xa = \Xa \setminus \int \Xa $. Thus, each $\Xa$-gap is also the $\Ya$-gap, and the interior of an $\Xa$-interval is the $\Ya$-gap.

\begin{thm}\label{13}
$S(\Ya)= \{0\} \cup \{\frac1{4^n}\colon n\in \omega \}$, i.e., the center of distances of the set $\Ya$ consists of exactly of the terms of the sequence $0,\frac14,\frac1{16},\ldots,\frac1{4^n},\ldots$.
\end{thm}

\begin{proof}
Since the numbers $0, \frac14, \frac13(=\sum_{n>0} \frac5{4^{2n}}), \frac12, \frac{17}{32}$ and $1 $ are in $\Ya$, we get
\[\textstyle
\bigcup \{(\frac1{4^{n+1}}, \frac1{4^n})\colon n \in \omega \} \cap S(\Ya) = \emptyset,
\]
as in the proof of Theorem \ref{12}.
We see that $1\in S(\Ya)$, because
\[\textstyle
(\Ya \cap [0,\frac23]) + 1 = \Ya \cap [1, \frac53].
\]
Moreover
\[\textstyle
(\Ya \cap [0,\frac16] +\frac14) \cup ( \Ya \cap [0, \frac16] + \frac12) \subset \Ya,
\] 
so $\frac14 \in S(\Ya)$. Similarly, we check that $\frac1{4^n} \in S(\Ya)$.
\end{proof}

\begin{cor}\label{c14}
Neither $\Za$ nor $\Ya$ is the set of subsums of a sequence.
\end{cor}

\begin{proof}
Because $S(\Za)=\{0\}$, thus Proposition \ref{I} decides the case with $\Za$. Also, this proposition decides the cases with $\Ya$, since $\frac53 \in \Ya$ and $\sum_{n\in \omega}\frac1{4^n}=\frac43$.
\end{proof}

Let us add that the set of subsums of the sequence $\{\frac1{4^n}_{n \in \omega}$ is included in $\Ya$. One can check this, observing that each number $\sum_{n\in A}\frac1{4^n}$, where the nonempty set $A\subset \omega$ is finite, is the right end of an $\Xa$-interval.

\section{Digital representation of points in the Cantorval $\Xa$}

Assume that $A=\{a_n\}_{ n>0}$ and $B=\{b_n\}_{ n>0}$ are digital representations of a point $x\in \Xa, $ i.e., \[ \textstyle \sum_{n>0}\frac{a_n}{4^n} = \sum_{n>0}\frac{b_n}{4^n}=x ,\] where $a_n, b_n \in \{0,2,3,5\}.$ 
We are going to describe dependencies between $ a_n$ and $ b_n $.
Suppose $n_0$ is the least index such that $a_{n_0}\not=b_{n_0}$.
Without loss of generality, we can assume that $a_{n_0}=2 <b_{n_0}=3$, bearing in mind that $\Xe(n_0)=\frac53\cdot\frac1{4^{n_0}}$.
And then we say that $A$ is \textit{chasing} $B$ (or $B$ is being \textit{caught} by $A$) in the $n_0$-step: in other words, $\sum_{i >n_0}\frac{a_i}{4^{i}} = \sum_{i >n_0}\frac{b_i}{4^{i}}+ \frac1{4^{n_0}}$ and $a_k=b_k$ for $k<n_0$.
If it is never the case that $a_k=5$ and $b_k=0$, then it has to be $b_k+3=a_k$ for all $k>n_0$.
In such a case, we obtain $ \sum_{i >n_0}\frac{a_i}{4^{i}} = \sum_{i >n_0}\frac{b_i}{4^{i}}+ \frac1{4^{n_0}}, $ since $ \sum_{i >n_0}\frac{3}{4^{i}} = \frac1{4^{n_0}}$.

Suppose $n_1$ is the least index such that $a_{n_1}=5$ and $b_{n_1}=0, $ thus $B$ is chasing $A$ in the $n_1$-step. Proceeding this way, we obtain an increasing (finite or infinite) sequence $n_0 < n_1 < \ldots$ such that $|a_{n_k} - b_{n_k}| = 5$ and $|a_i -b_i| =3$ for $n_0 < i \notin \{n_1, n_2 , \ldots \}$. Moreover, $A$ starts chasing $B$ in the $n_k$-step for even $k$'s and $B$ starts chasing $A$ in the $n_k$-step for odd $k$'s, for the rest of steps changes of chasing do not occur.

\begin{thm}\label{t11}
Assume that $x\in \Xa$ has more than one digital representation. There exists the finite or infinite sequence of positive natural numbers $n_0 < n_1 < \ldots$ and exactly two digital representations $\{a_n\}_{n>0}$ and $\{b_n\}_{n>0}$ of $x$ such that:
\begin{itemize}
\item $a_k =b_k$, 
as far as $0<k<n_0$;
\item $a_{n_0}= 2$ and $b_{n_0}=3$;
\item $a_{n_k}= 5$ and $b_{n_k}=0$, for odd $k$;
\item $a_{n_k}= 0$ and $b_{n_k}=5$, for even $k>0$;
\item $a_i \in \{3,5\}$ and $a_i - b_i =3$, as far as $n_{2k} < i < n_{2k+1}$; 
\item $a_i \in \{0,2\}$ and $b_i - a_i =3$, as far as $n_{2k+1} < i < n_{2k+2}$.
\end{itemize}
\end{thm}

\begin{proof}
According to the chasing algorithm described above in step $n_k $ the roles of chasing are reversed. But, if the chasing algorithm does not start, then the considered point has a unique digital representation.
\end{proof}

The above theorem makes it easy to check the uniqueness of digital representation.
For example, if $x \in \Xa$ has a digital representation $\{x_n\}_{n>0}$ such that $x_n=2$ and $x_{n+1}=3$ for infinitely many $n$, then this representation is unique.
Indeed, suppose $A=\{a_n\}_{n>0}$ and $B=\{b_n\}_{n>0}$ are two different digital representations of a point $x\in \Xa $ such that $a_k =b_k$, as far as $0 <k<n_0$ and $a_{n_0}= 2 < b_{n_0}=3$.
By Theorem \ref{t11}, never the digit 3 occurs immediately after the digit 2 in digital representations of $x$ for digits greater than $n_0$, since it has to be $|a_k -b_k| >2$ for $k>n_0$.

The map $\{a_n\}_{n>0}\mapsto\sum_{i>0}\frac{a_i}{4^{i}}$ is a continuous function from the Cantor set (a homeomorphic copy of the Cantor ternary set) onto the Cantorval such that the preimage of a point has at most two points.
In fact, the collection of points with two-point preimages and its complement are both of the cardinality continuum.
By the algorithm described above, each sequence $n_0 < n_1 < \ldots $ of positive natural numbers determines exactly two sequences $\{a_n\}_{n>0}$ and $\{b_n\}_{n>0}$ such that \[ \textstyle \sum_{n>0}\frac{a_n}{4^n} = \sum_{n>0}\frac{b_n}{4^n},\] and vice versa.
In the following theorem, we will use the abbreviation $B=\{\frac2{4^n}\colon n>0\}\cup\{\frac3{4^n}\colon n>0\}$. 

\begin{thm}\label{14}
Let $A\subset B$ be such that $B\setminus A$ and $A$ are infinite.
Then the set of subsums of a sequence consisting of different elements of $A$ is homeomorphic to the Cantor set.
\end{thm}

\begin{proof}
Fix a non-empty open interval $\Ia$.
Assume that $\{a_n\}_{n>0}$ is the digital representation of a point  $\sum_{n>0}\frac{a_{n}}{4^n} \in \Ia$.
Choose natural numbers $m>k$ such that numbers $\sum_{n=1}^k \frac{a_{n}}{4^n}$ and $\frac5{4^m} +\sum_{n=1}^k \frac{a_{n}}{4^n}$ belong to $\Ia$.
Then choose	 $j>m$ so that $\frac{a}{4^{j}} \in B \setminus A$, where $a =2$ or $a=3$.
Finally put
$b_{n} = a_{n}$,  for $ 0<n \leqslant k$; $b_{j} = 5$; $b_{j+1}=2$ and $b_{i} = 0$ for other cases.
Since $b_m=0$, we get  $\sum_{i>0}\frac{b_{i}}{4^i} \in \Ia$.
Theorem \ref{t11} together with conditions $b_j=5$ and $b_{j+1}=2$ imply that the point  $\sum_{i>0}\frac{b_i}{4^i}$ is not in the set of subsums of $A$.
Thus, this set being dense in itself and closed is homeomorphic to the Cantor set.
\end{proof}

\end{document}